\documentclass[a4,12pt]{article}
\usepackage{amsmath, amssymb, amsfonts, amstext, amsthm, textcomp}
\usepackage[mathscr]{euscript}
\newtheorem{thm}{Theorem}[section]
\newtheorem{lem}{Lemma}[section]

\newtheorem{defn}{Definition}[section]

\newtheorem{ex}{Example}[section]
\begin{document}
\noindent

\begin{center}
 {\bf \large Pseudo Schur Complements, Pseudo Principal Pivot Transforms and Their Inheritance Properties}
\vspace{.5cm}

             { \bf Kavita Bisht}\\
Department of Mathematics\\
                         Indian Institute of Technology Madras\\

                               Chennai 600 036, India \\
            
                and\\
              {\bf K.C. Sivakumar} \\
                         Department of Mathematics\\
                         Indian Institute of Technology Madras\\

                               Chennai 600 036, India. \\
\end{center}

\begin{abstract}
 In this short note, we prove some basic results on pseudo Schur complement and the pseudo principal pivot
transform of a block matrix. Pseudo Schur complement and pseudo principal pivot ransform are extensions of the Schur complement and the principal pivot transform, respectively, where
the usual inverse is replaced by the Moore-Penrose inverse. The objective is to record these results for use in future.
\end{abstract}
\section{Introduction}

Let $\mathbb{R}^{m \times n}$ denote the set of all real matrices of order $m \times n$, $\mathbb{R}^{n}$ denote the n dimensional real Euclidean space and $\mathbb{R}^{n}_{+}$ denote the nonnegative orthant in $\mathbb{R}^{n}$.  For a matrix $M\in \mathbb{R}^{m\times n}$, we denote the null space and the transpose of $M$ by $N(M)$ 
and $M^T$, respectively. The Moore-Penrose inverse of a matrix $M \in \mathbb{R}^{m\times n}$, 
denoted by $M^\dagger$ is the unique solution $X\in \mathbb{R}^{n\times m}$ of the equations: 
$M=MXM, X=XMX, (MX)^T=MX$ and $(XM)^T=XM$. If $M$ is nonsingular, then of course, we have 
$M^{-1}=M^\dagger$. Recall that $M\in \mathbb{R}^{n \times n}$ is called range-symmetric 
(or an EP matrix) if $R(M^T)=R(M)$. For this class of matrices, the Moore-Penrose inverse $A^{\dagger}$
commutes with $A$. Let $M \in \mathbb{R}^{m\times n}$ be a block matrix partitioned as 
\begin{center}
 $\left( \begin{array}{lr} A  & B \\ C &  D  \end{array}
\right)$
\end{center}
where $A \in \mathbb{R}^{k \times k}$ is nonsingular. Then the classical Schur 
complement of $A$ in $M$ denoted by $M/A$ is given by $F=D-CA^{-1}B \in \mathbb{R}^{(m-k) \times (n-k)}$. 
This notion has proved to be a fundamental idea in many applications like numerical analysis,
statistics and operator inequalities, to name a few. This expression for the Schur complement was fruther extended by Carlson \cite{carl} to 
include all matrices of the form $D-CA^{\{1\}}B$, where $A^{\{1\}}$ denotes any arbitrary 
$\{1\}$-inverse of $A$ (A $\{1\}$-inverse of $A$ is a any matrix $X$ which satisfies $AXA=A$). 
Carlson proved that this generalized Schur complement is invariant under the choice of $A^{\{1\}}$ 
if and only if $B=0$ or $C=0$ or $R(B) \subseteq R(A)$ and $R(C^T) \subseteq R(A^T)$. He studies the 
relationship of the generalized Schur complements to certain optimal rank problems. 
The expression $D-CA^{\dagger}B$, is also referred to in the literature as the generalized 
Schur complement \cite{carlhaymark}, where the Sylvester's determinantal formula and a quotient 
formula are proved, among other things. Nevertheless, since we will be concerned with the 
case of the Moore-Penrose inverse (which is also called the pseudo inverse), we shall 
refer to it as the pseudo Schur complement.\\
Again, consider $M$, 
partitioned as above. If $A$ is nonsingular, then the principal pivot transform (PPT) of 
$M$ is the block matrix defined by 
\begin{center} 
\[ \left( \begin{array}{cc}
A^{-1} & -A^{-1}B \\
CA^{-1} & F
\end{array} \right)\] 
\end{center}
where $F$ is again, the Schur complement $F=D-CA^{-1}B$. For an excellent survey of PPT we refer the reader to \cite{tsat}. Just as in the case of the 
generalized Schur complement, it is natural to study the PPT when the usual inverses are replaced 
by generalized inverses. Meenakshi \cite{meen}, was perhaps the first to study such a generalization 
for the Moore-Penrose inverse.\\
The principal pivot transform involving the Moore-Penrose inverse has been studied in the literature. In what follows, 
we consider it once again, albeit with a different name, the pseudo principal pivot transform. We also find it 
natural to consider the complementary pseudo principal pivot transform. 

\begin{defn}\label{D: D1}  Let $M$ be defined as above. Then the pseudo principal pivot transform of $M$ relative to $A$ is defined by\\
\begin{center}
$H:=pppt(M,A)_{\dagger}=
\left(\begin{array}{cc}
  A^\dagger & -A^\dagger B \\
  CA^\dagger  & F
 \end{array}\right)$,
\end{center} where $F=D-CA^\dagger B$. The complementary pseudo principal pivot transform of $M$ relative to $D$ is defined by\\
\begin{center}
$J:=cpppt(M,D)_{\dagger}=
\left(\begin{array}{cc}
 G & BD^\dagger \\
-D^\dagger C & D^\dagger
\end{array}\right)$, 
\end{center} 
 where $G=A-BD^\dagger C$. 
\end{defn}

We now prove two extensions of the domain-range exchange property, well known in the nonsingular case. 

\begin{lem}\label{L2}
(i) Suppose that $R(B)\subseteq R(A)$ and $R(C^T)\subseteq R(A^T)$.
Then $M$ and $H=pppt(M,A)_{\dagger}$ are related by the formula:\\
\begin{center}
$M
\left(\begin{array}{cc}
x^{1}\\
x^{2}
\end{array}\right)
=
\left(\begin{array}{cc}
AA^\dagger y^{1} \\
y^{2}
\end{array}\right)$
if and only if $H
\left(\begin{array}{cc}
  y^{1} \\
  x^{2}
 \end{array}\right)
=
\left(\begin{array}{cc}
  A^\dagger A x^{1} \\
  y^{2}
 \end{array}\right)$.\\
\end{center}
(ii) Suppose that $R(C)\subseteq R(D)$ and $R(B^T)\subseteq R(D^T)$.  
Then $M$ and $J=cpppt(M,D)_{\dagger}$ are related by the formula:
\begin{center}
$M
\left(\begin{array}{cc}
x^{1}\\
x^{2}
\end{array}\right)
=
\left(\begin{array}{cc}
y^{1}\\
DD^\dagger y^{2}
\end{array}\right)$
if and only if
$J
\left(\begin{array}{cc}
y^{1}\\
x^{2}
\end{array}\right)
=
\left(\begin{array}{cc}
x^{1}\\
D^\dagger Dy^{2}
\end{array}\right)$.
\end{center} 
\end{lem}
\begin{proof}
We prove (i). The proof for (ii) is similar. Suppose that $M
\left(\begin{array}{cc}
x^{1}\\
x^{2}
\end{array}\right)
=
\left(\begin{array}{cc}
AA^\dagger y^{1}\\
y^{2}
\end{array}\right)$. Then $Ax^{1}+Bx^{2}=AA^\dagger y^{1}$ and $Cx^{1}+Dx^{2}=y^{2}$. Premultipling the first equation by 
$A^\dagger$ (and rearranging) we get $A^\dagger y^{1}-A^\dagger Bx^{2}=A^\dagger Ax^{1}$. Premultiplying this equation by 
$C$, we then have  $CA^\dagger y^{1}-CA^\dagger Bx^{2}=CA^\dagger Ax^{1}=Cx^{1}$. So,  
$CA^\dagger y^{1}+Fx^{2}
=CA^\dagger y^{1}+Dx^{2}-CA^\dagger Bx^{2}
=Cx^{1}+Dx^{2}
=y^{2}$. Thus, $H
\left(\begin{array}{cc}
y^{1}\\
x^{2}
\end{array}\right)
=
\left(\begin{array}{cc}
A^\dagger y^{1}-A^\dagger Bx^{2}\\
CA^\dagger y^{1}+Fx^{2}
\end{array}\right)
=
\left(\begin{array}{cc}
A^\dagger Ax^{1}\\
y^{2}
\end{array}\right)$.\\
Conversely, let $H
\left(\begin{array}{cc}
y^{1}\\
x^{2}
\end{array}\right)
=
\left(\begin{array}{cc}
A^\dagger Ax^{1}\\
y^{2}
\end{array}\right)$. Then $A^\dagger y^{1}-A^\dagger Bx^{2}=A^\dagger Ax^{1}$ and $CA^\dagger y^{1}+(D-CA^\dagger B)x^{2}=y^{2}$. 
Premultiplying the first equation by $A$, we have $AA^\dagger y^{1}-Bx^{2}=Ax^{1}$ so that $Ax^{1}+Bx^{2}=AA^\dagger y^{1}$. 
Again, premultiplying the first equation by $C$, we get $CA^\dagger y^{1}-CA^\dagger Bx^{2}=Cx^{1}$. Hence, using the 
second equation we have, $Cx^{1}+Dx^{2}=CA^\dagger y^{1}-CA^\dagger Bx^{2}+Dx^{2}=y^{2}$, proving that $M
\left(\begin{array}{cc}
x^{1}\\
x^{2}
\end{array}\right)
=
\left(\begin{array}{cc}
AA^\dagger y^{1}\\
y^{2}
\end{array}\right)$.
\end{proof}

Let $M=\left(\begin{array}{cc}
A & B\\
C & D
\end{array}\right)$ with $R(B)\subseteq R(A)$ and $R(C^T)\subseteq R(A^T)$. First, in the following example, 
we show that $H^\dagger \neq J$ in general.  

\begin{ex}
Let $A=\left(\begin{array}{cc}
1 & -1\\
2 & -2
\end{array}\right)$, $B=\left(\begin{array}{cc}
1\\
2
\end{array}\right)$, $C=\left(\begin{array}{cc}
-1 & 1
\end{array}\right)$ and $D=\left(\begin{array}{c}
0
\end{array}\right)$. Then $M=\left(\begin{array}{ccc}
1 & -1 & 1\\
2 & -2 & 2\\
-1 & 1 & 0
\end{array}\right)$, $R(B)\subseteq R(A)$, $R(C^T)\subseteq R(A^T)$, $H=pppt(M,A)_{\dagger}=\frac{1}{10}\left(\begin{array}{ccc}
1 & 2 & -5\\
-1 & -2 & 5\\
-2 & -4 & 10
\end{array}\right)$, $J=cpppt(M,D)_{\dagger}=\left(\begin{array}{ccc}
1 & -1 & 0\\
2 & -2 & 0\\
0 & 0 & 0
\end{array}\right)$ and $H^{\dagger}=\left(\begin{array}{ccc}
1 & -1 & 0\\
2 & -2 & 0\\
-1 & 1 & 1
\end{array}\right)\neq J$.\\
\end{ex}

Now, we give the necesarry conditions under which $H^\dagger=J$. Once again, the natural conditions are handy.

\begin{thm} \label{pptcppt}
Let $A \in \mathbb{R}^{m \times n}, B \in \mathbb{R}^{m \times p}, C \in \mathbb{R}^{s \times n}, D\in \mathbb{R}^{s \times p}$ and $M=
\left(\begin{array}{cc}
  A & B \\
  C & D
 \end{array}\right)$. Suppose that $R(B)\subseteq R(A)$, $R(C^T)\subseteq R(A^T)$, $R(C)\subseteq R(D)$ and $R(B^T)\subseteq R(D^T)$. 
Then $H^\dagger=J$, where $H=pppt(M,A)_{\dagger}$ and $J=cpppt(M,D)_{\dagger}$.\\
\end{thm}
\begin{proof}
$J=\left(\begin{array}{cc}
G & BD^\dagger\\
-D^\dagger C & D^\dagger
\end{array}\right)$ and so, 
\begin{center}
$JH=\left(\begin{array}{cc}
GA^\dagger+BD^\dagger CA^\dagger & -GA^\dagger B+BD^\dagger F\\
0 & D^\dagger CA^\dagger B+D^\dagger F
\end{array}\right)$.
\end{center}
We have, 
\begin{center} 
$GA^\dagger+BD^\dagger CA^\dagger=(G+BD^\dagger C)A^\dagger =AA^\dagger$,
\end{center}
$-GA^\dagger B+BD^\dagger F=-(A-BD^\dagger C)A^\dagger B+BD^\dagger(D-CA^\dagger B)=-B+BD^\dagger D=0$. 
Also, $D^\dagger CA^\dagger B+D^\dagger F=D^\dagger(F+CA^\dagger B)=D^\dagger D$. So, $JH=\left(\begin{array}{cc}
AA^\dagger & 0\\
0 & D^\dagger D
\end{array}\right)$.
Thus $(JH)^T=JH$. Also, $JHJ=\left(\begin{array}{cc}
AA^\dagger G & AA^\dagger BD^\dagger\\
-D^\dagger C & D^\dagger
\end{array}\right)=\left(\begin{array}{cc}
G & BD^\dagger\\
-D^\dagger C & D^\dagger
\end{array}\right)$, since $AA^\dagger B=B$ and $AA^\dagger G=AA^\dagger(A-BD^\dagger C)=A-BD^\dagger C=G$.\\
Next, \begin{eqnarray*}
 HJH &=&\left(\begin{array}{cc}
A^\dagger & -A^\dagger BD^\dagger D\\
CA^\dagger & FD^\dagger D
\end{array}\right)\\
&=&\left(\begin{array}{cc}
A^\dagger & -A^\dagger B\\
CA^\dagger & F
\end{array}\right),
\end{eqnarray*}
since $FD^\dagger D=(D-CA^\dagger B)D^\dagger D=D-CA^\dagger BD^\dagger D=D-CA^\dagger B=F$. 
Also, 
\begin{center}
$HJ=\left(\begin{array}{cc}
A^\dagger G+A^\dagger BD^\dagger C & 0\\
CA^\dagger G-FD^\dagger C & CA^\dagger BD^\dagger+ FD^\dagger
\end{array}\right)$.
\end{center}
We have 
\begin{center}
$A^\dagger G+A^\dagger BD^\dagger C=A^\dagger(G+BD^\dagger C) =A^\dagger A$,
\end{center}
\begin{center}
 $CA^\dagger G-FD^\dagger C=CA^\dagger(A-BD^\dagger C)-(D-CA^\dagger B)D^\dagger C=C-DD^\dagger C=0$,
\end{center}
since $R(C)\subseteq R(D)$. Finally, $CA^\dagger BD^\dagger+FD^\dagger=(CA^\dagger B+F)D^\dagger =DD^\dagger$. 
So, $HJ=\left(\begin{array}{cc}
A^\dagger A & 0\\
0 & DD^\dagger
\end{array}\right)$, so that $(HJ)^T=HJ$.
\end{proof}

It is well known that the Schur complement and formulae for inverses of partitioned matrices go 
hand in hand. We proceed in the same spirit, where we first consider the Moore-Penrose 
inverse of partitioned matrices. The following result is quite well known. This 
has been proved in \cite{carlhaymark}. However, we 
provide an alternative proof for the sake of completeness and ready reference. 
\begin{thm}\label{MPA}
Let $A \in \mathbb{R}^{m \times n}, B \in \mathbb{R}^{m \times p}, C \in \mathbb{R}^{s \times n}, D\in \mathbb{R}^{s \times p}$ and $M=
\left( \begin{array}{cc} A & B \\ C & D \end{array} \right)$. Suppose that $ R(C^T)\subseteq R(A^T)$, $R(B)\subseteq R(A)$, $R(CA^\dagger)\subseteq R(F)$ and 
$R((A^\dagger B)^T)\subseteq R(F^T)$, where $F=D-CA^\dagger B$. Then
\begin{center}
$M^\dagger=\left( \begin{array}{cc}
A^\dagger +A^\dagger BF^\dagger CA^\dagger & -A^\dagger BF^\dagger \\
-F^\dagger CA^\dagger & F^\dagger
\end{array} \right)$ . 
\end{center} 
\end{thm}
\begin{proof}
First, we observe that $CA^\dagger A=C$ (since $R(C^T)\subseteq R(A^T)$) and 
$A^\dagger B=A^\dagger BFF^\dagger$ (since $R((A^\dagger B)^T)\subseteq R(F^T)$). For $x\in \mathbb{R}^{p}$, let $y=B(I-F^\dagger F)x$.
Then $y \in R(B) \subseteq R(A)$. Also,
$A^\dagger y= A^\dagger B(I-F^\dagger F)x=0 $ so that $y \in N(A^\dagger)=N(A^T)$.
This means that $y=0$ and so $BF^\dagger F=B$. We then have $DF^\dagger F=(F+CA^\dagger B)F^\dagger F = F+CA^\dagger BF^\dagger F
                 = F+CA^\dagger B
                 = D$.
Also, 
$A^{\dagger} B + A^{\dagger} B F^{\dagger} CA^{\dagger}B-A^{\dagger}BF^\dagger D=A^\dagger B + A^\dagger BF^\dagger (CA^\dagger B- D) =A^\dagger B-A^\dagger
BF^\dagger F=0$.\\
Set $X=
\left(\begin{array}{cc}
  A^\dagger+A^\dagger BF^\dagger CA^\dagger & -A^\dagger BF^\dagger \\
  -F^\dagger CA^\dagger & F^\dagger
 \end{array}\right)$. Then
\begin{center}
$XM =\left(\begin{array}{cc}
  A^\dagger + A^\dagger BF^\dagger CA^\dagger & -A^\dagger BF^\dagger \\
  -F^\dagger CA^\dagger & F^\dagger
 \end{array}\right)$
$\left(\begin{array}{cc}
  A & B \\
  C & D
 \end{array}\right)=
\left(\begin{array}{cc}
  A^\dagger A & 0 \\
  O & F^\dagger F
 \end{array}\right)$.
\end{center}
So, $(XM)^T=XM$.
Also,
$MXM=
\left(\begin{array}{cc}
  A & BF^\dagger F \\
  C & DF^\dagger F
 \end{array}\right)=\left(\begin{array}{cc}
A & B\\
C & D
\end{array}\right)=M$. Further, \\
\begin{eqnarray*}
 XMX 
&=&
\left(\begin{array}{cc}
  A^\dagger A & 0 \\
  O & F^\dagger F
 \end{array}\right)
\left(\begin{array}{cc}
  A^\dagger+A^\dagger BF^\dagger CA^\dagger & -A^\dagger BF^\dagger \\
  -F^\dagger C A^\dagger & F^\dagger
 \end{array}\right) \\
&=&
\left(\begin{array}{cc}
  A^\dagger+A^\dagger BF^\dagger CA^\dagger & -A^\dagger BF^\dagger \\
  -F^\dagger C A^\dagger & F^\dagger
 \end{array}\right)\\
&=& X. 
\end{eqnarray*}\\
Finally,\\
\begin{center}
$MX=
\left(\begin{array}{cc}
  A & B \\
  C & D
 \end{array}\right)
\left(\begin{array}{cc}
  A^\dagger+A^\dagger BF^\dagger CA^\dagger & -A^\dagger BF^\dagger \\
  -F^\dagger C A^\dagger & F^\dagger
 \end{array}\right)
=
\left(\begin{array}{cc}
  AA^\dagger & 0 \\
  0 & FF^\dagger
 \end{array}\right)$,
\end{center} where we have used the facts that $AA^\dagger B=B$ (since $R(B)\subseteq R(A)$) and 
$ CA^\dagger=FF^\dagger CA^\dagger$ (since $R(CA^\dagger)\subseteq R(F)$). Clearly, $(MX)^T=MX$, completing the proof.
\end{proof}
Next, we illustrate the above theorem with the help of example.
\begin{ex}
Let $A=\left(\begin{array}{cc}
1 & -1\\
2 & -2
\end{array}\right)$, $B=\left(\begin{array}{cc}
1 & -2\\
2 & -4
\end{array}\right)$, $C=\left(\begin{array}{cc}
1 & -1\\
-1 & 1
\end{array}\right)$ and $D=\left(\begin{array}{cc}
1 & 1\\
0 & 0
\end{array}\right)$.
Then $M=\left(\begin{array}{cccc}
1 & -1 & 1  & -2\\
2 & -2 & 2 & -4\\
1 & -1 & 1 & 1\\
-1 & 1 & 0 & 0
\end{array}\right)$, $R(B)\subseteq R(A)$, $R(C^T)\subseteq R(A^T)$, $R(CA^\dagger)\subseteq R(F)$ and $R((A^\dagger B)^T)\subseteq R(F^T)$.
\end{ex}

\begin{center}
$M^\dagger=\left(\begin{array}{cc}
 A^\dagger+A^\dagger BF^\dagger CA^\dagger & -A^\dagger BF^\dagger \\
  -F^\dagger C A^\dagger & F^\dagger
 \end{array}\right)=\left(\begin{array}{cccc}
 0 & 0 & 0 & \frac{-1}{2}\\
0 & 0 & 0 & \frac{1}{2}\\
\frac{1}{15}& \frac{2}{15} & \frac{2}{3} & 1\\
\frac{-1}{15} & \frac{-2}{15} & \frac{1}{3} & 0
\end{array}\right)$.
\end{center}
 Next, we state a complementary result. This does 
not seem to be as well known as the previous result. However, we skip its proof. Note that this result 
uses the pseudo Schur complement $G=A-BD^\dagger C$, which is called the complementary 
Schur complement. This time, the natural conditions are $R(B^T)\subseteq R(D^T)$ 
and $R(C)\subseteq R(D)$. These conditions guarantee that the complementary Schur complement 
$G=A-BD^{\{1\}} C$ is invariant under any $\{1\}$-inverse $D^{\{1\}}$ of $D$.

\begin{thm}\label{MPD}
Let $M=
\left(\begin{array}{cc}
  A & B \\
  C & D
 \end{array}\right)$ with the blocks defined as earlier. Suppose that $R(B^T)\subseteq R(D^T)$, 
$R(C)\subseteq R(D)$, $R(BD^\dagger)\subseteq R(G)$ and $R((D^\dagger C)^T)\subseteq R(G^T)$, where
   $G=A-BD^\dagger C$. Then 
\begin{center}
$M^\dagger=
 \left(\begin{array}{cc}
  G^\dagger & -G^\dagger BD^\dagger \\
  -D^\dagger C G^\dagger & D^\dagger+D^\dagger CG^\dagger BD^\dagger
\end{array}\right)$. 
\end{center}
\end{thm}
By comparing the two expressions for $M^\dagger$, we obtain the formulae:
\begin{center}
$G^\dagger =A^\dagger+A^\dagger BF^\dagger CA^\dagger$ and $F^\dagger=D^\dagger+D^\dagger CG^\dagger BD^\dagger$,
\end{center}
in the presence of all the eight inclusions of Theorem \ref{MPA} and Theorem \ref{MPD}. Using these 
formulae, next we obtain another expression for the Moore-Penrose inverse of $M$ involving the pseudo 
Schur complements of $A$ and $D$. 

\begin{thm}\label{MPAD}
Let $M=
\left(\begin{array}{cc}
  A & B \\
  C & D
 \end{array}\right)$. Suppose that $ R(C^T)\subseteq R(A^T)$, $R(B)\subseteq R(A)$, $R((A^\dagger B)^T)\subseteq R(F^T)$, 
$R(CA^\dagger)\subseteq R(F)$, $R(C)\subseteq R(D)$, $R(B^T)\subseteq R(D^T)$, $R(BD^\dagger)\subseteq R(G)$ and 
$R((D^\dagger C)^T)\subseteq R(G^T)$, where
$F=D-CA^\dagger B$ and $G=A-BD^\dagger C$. Then
\begin{center}
$M^\dagger=
\left(\begin{array}{cc}
G^\dagger & -A^\dagger BF^\dagger\\
-D^\dagger CG^\dagger & F^\dagger
\end{array}\right)$.
\end{center}
\end{thm}


\begin{thebibliography}{10}

\bibitem{carl} 
D. Carlson, {\it What are Schur complements, anyway?}, Lin. Alg. Appl., {\bf 74}, (1986) 257-275.
\bibitem{carlhaymark}
D. Carlson, E.V. Haynsworth and T.L. Markham, {\it A generalization of the Schur complement by means of the Moore-Penrose inverse}, SIAM J. Appl. Math, {\bf 26} (1974) 169-175.


\bibitem{meen}
A.R. Meenakshi, {\it Principal pivot transforms of an $EP$ matrix}, C.R. Math. Rep. Acad. Sci. Canada, {\bf8}, (1986) 121-126.

\bibitem{tsat}
M. Tsatsomeros, {\it Principal pivot transforms: Properties and applications}, Lin. Alg. Appl., {\bf 307} (2000) 151-165.

\bibitem{zang}
F. Zhang, {\it The Schur Complement and Its Applications}, Springer, New York, 2005.

\end{thebibliography}
\end{document}